\PassOptionsToPackage{numbers}{natbib}
\documentclass[graybox,natbib,envcountsame]{svmult}

\usepackage{mathptmx}
\usepackage{helvet}
\usepackage{courier}
\usepackage{makeidx}
\usepackage{graphicx}
\usepackage{multicol}
\usepackage[bottom]{footmisc}

\usepackage{amsmath}   
\usepackage{amssymb} 
\usepackage[utf8]{inputenc}
\usepackage{tikz}
\usetikzlibrary{%
  cd, 
  matrix 
}

\usepackage{mathrsfs}
\usepackage[all]{xy}

\newcommand{\R}{\mathbb{R}}

\spnewtheorem*{AAT}{Abel's Addition Theorem}{\bfseries}{\itshape}

\smartqed
\usepackage{etoolbox}
\patchcmd{\qed}{\ifmmode\qedsymbol}{\ifmmode\the\qedsymbol}{}{\foobar}
\patchcmd{\qed}{\hfil\qedsymbol}{\hfil\the\qedsymbol}{}{\foobar}
\qedsymbol{\proofSymbol}

\newcommand{\qedhere}{\tag*{\the\qedsymbol}}

\begin{document}

\title*{On a  quasi-resonant nonlinear Schr\"odinger equation on  irrational tori}

\author{Alexander Hrabski, Yulin Pan, Gigliola Staffilani \and Bobby Wilson}
\institute{A.~Hrabski \at University of Michigan, Department of Naval Architecture and Marine Engineering, Ann Arbor, MI 48109, USA,
  \email{ahrabski@umich.edu} \and %
  Y.~Pan \at University of Michigan, Department of Naval Architecture and Marine Engineering, Ann Arbor, MI 48109, USA, \email{yulinpan@umich.edu} \and %
  G.~Staffilani \at Massachusetts Institute of Technology, Department of Mathematics, Cambridge, MA 02139, USA, \email{gigliola@math.mit.edu} \and %
  B.~Wilson \at University of Washington, Department of Mathematics, Box 354350,  Seattle, WA 98195-4350, USA, \email{blwilson@uw.edu}}




\maketitle

\abstract{
    In this work we consider a 2D    Schr\"odinger  initial value problem  in which the cubic nonlinearity has been restricted to quasi-resonant interactions,  moreover we consider   periodic solutions whose ratio between the periods is an algebraic irrational positive number. We  analyze the  energy transfer  for these solutions  by estimating how the support in frequency space of an initial datum propagates in time. Moreover we complement the analytic  study with numerical experimentation.  As a byproduct of our investigation we  prove that the quasi-resonant  cubic Schr\"odinger  initial value problem we consider, in both the focusing and defocusing case,  are globally well-posed for initial data of finite mass. Finally using a newly discovered conservation law we also show that if the initial data of the  quasi-resonant  initial value problem is compact, then any  norm $H^s, \, s\geq 1,$ of the associate solution is uniformly bounded.
}


\section{Introduction}
In this work we continue the study of energy transfer for solutions to  defocusing  and periodic nonlinear Schr\"odinger (NLS) equations. Here instead of focusing our attention to obtain the best estimates for  the  magnitude of these solutions as time evolves, we analyze the dynamics of their   support in the frequency space.  This study can be viewed under the much wider umbrella of weak turbulence theory and was initiated in the form that inspired us for this work  by Bourgain in \cite{bourgain1996growth}. In weak turbulence theory, given a solution $\psi(t,x)$ to a dispersive equation, one is interested in understanding the long time behavior of the Fourier coefficients $|\hat \psi(t,k)|^2$, and in particular the migration of their Fourier support to higher frequencies $k$. In order to detect this migration, one may study the asymptotic behavior in time of higher Sobolev norms of the solution $\psi(t,x)$. Several works have appeared in recent years concerning both the proof of   polynomial in time upper bounds  for higher Sobolev norms of solutions \cite{Staffilani1997, Soinger2011, Soinger2011-1, Soinger2012,
CKS2012, deng2019growth, deng2019growth-only, PTV17, Wang2008-1, Wang2008}, and the constructions of solutions that exhibit a growth \cite{colliander2010transfer, carles2010energy, Hani2015, Hani2019, Haus2020, Haus2016,  Haus2015,  GG}. For some periodic linear  equations   with potential some of these time bounds have been proved to be almost sharp \cite{Bourgain2003, Bourgain1999, Bourgain1999-1}, namely that they could be at most logarithmic. But in general    one is  far from obtaining sharp results except when the domain is a mixed manifold as in \cite{Hani2015}, or when one considers irrational tori as we will show below. In this work we are interested in analyzing both analytically and numerically how the question of energy transfer may depend on the relative rationality of the periods of the solution at hand. In particular here we consider a
2D set up involving a  cubic NLS in which the nonlinear frequency interactions are restricted to be quasi-resonant\footnote{A precise definition is given in Definition \ref{quasi-res}.}.

Let us then consider  $\underline{\omega}=(\omega_1, \omega_2) \in \mathbb{R}_+^2$ and $\mathbb{T}_{\underline{\omega}}^2= \mathbb{R}/\omega_1\mathbb{Z} \times \mathbb{R}/\omega_2\mathbb{Z}$  the  two-dimensional flat torus scaled by $\underline{\omega}$. If   $\underline{\omega}$ satisfies
	\begin{align}\label{irr}
		\frac{\omega_1^2}{\omega_2^2} \not\in \mathbb{Q}
	\end{align}
	 we  stipulate that the torus is  {\it irrational}, otherwise that is {\it rational}.

Let us then  consider the NLS
	\begin{align} \label{irrNLS}
		\begin{array}{ll}
			i \partial_t \psi = \Delta \psi - |\psi|^2\psi, & x \in \mathbb{T}_{\underline{\omega}}^2, \, t \in \mathbb{R}\\
			\psi(0) = \psi_0 \in H^s(\mathbb{T}_{\underline{\omega}}^2),
		\end{array}
	\end{align}
	where $H^s(\mathbb{T}_{\underline{\omega}}^2)$ represents the $L^2$ based Sobolev space of index $s$ with norm $\|\cdot\|_s$.
	In \cite{staffilani2020stability} the third and forth authors proved that if the torus is irrational then the construction in \cite{colliander2010transfer} and \cite{carles2010energy} of solutions for which high Sobolev norms are  mildly growing cannot be performed since the set of resonant  frequencies  is not large enough, and in particular it does not contain certain diamond configurations that are necessary as stepping stones in the dynamics of those constructions.  Still in \cite{GG} the authors proved that if the irrationality of the torus is not too severe, namely the ratio in \eqref{irr} is {\it close} to a rational number, still the result in \cite{colliander2010transfer} is valid. It is clear then that the dynamics of the periodic NLS cannot simply be reduced to the dichotomy of the domain being a rational or irrational torus. Going back to the  paper \cite{staffilani2020stability},  it was also remarked there that in a 2D irrational torus the four-waves resonance relation among the frequencies splits into two uncoupled 1D resonance relations, one per each component of the frequency vector.  We recall that in 1D  problem \eqref{irrNLS} is integrable, and as a consequence  solutions are  uniformly bounded in any Sobolev norm \cite{KST17}. In \cite{staffilani2020stability} it was  in fact conjectured that the 1D  uncoupled resonant interactions mentioned above and the 1D integrability may be a possible explanation for the expected weaker growth of the Sobolev norms in the irrational setting. We will go back to this conjecture later in this paper in Remark \ref{integrability}.     This weaker growth was also reported  in the  polynomial upper bounds  obtained in  \cite{deng2019growth-only, deng2019growth}, where the authors proved that for 3D irrational tori the polynomial bounds in time have a smaller degree compared to the rational case.
	
	Intuitively this should not come as a surprise since the irrationality of the torus should mitigate the boundary effects when the wave solution interacts with  the boundary itself.  In a sense the relative  irrationality of the periods should introduce some very faint dispersion. It is to be noted in fact that in $\mathbb{R}^2$, where dispersion can act without any obstacle,  recent results in \cite{Dodson2016} guarantee that the solution $\psi$ of \eqref{irrNLS} scatters in any $H^s$ norm for $s\geq 0$, and hence any Sobolev norm of solution $\psi$  is uniformly bounded in time.
	
	From the numerical experimentation point of view fewer results are at the moment available. We recall the work in \cite{Colliander-num}, where the authors studied in greater detail the construction in \cite{colliander2010transfer}, and the work in \cite{hrabski2020effect}, where the first two  authors of this paper  evaluated the stationary spectrum and energy cascade on rational and irrational tori for the 2D Majda-McLaughlin-Tabak equation (with high-frequency dissipation). One of their important findings is that in fact the energy cascade on irrational tori is much less efficient than on rational ones.

	 \medskip
	 We are now ready to state the main analytic results of this paper. 	 
We start with a  global well-posedness theorem   that is of independent interest.

\begin{theorem}\label{mainthm2} Assume that the torus $\mathbb{T}^2_{\underline{\omega}}$ is  a generic irrational ordered pair.
Consider the quasi-resonant focusing and defocusing NLS  initial value problem
\begin{align} \label{irrNLS-res}
		\begin{array}{ll}
			i \partial_t \psi = \Delta \psi \pm (|\psi|^2\psi)^*, & x \in \mathbb{T}_{\underline{\omega}}^2, \, t \in \mathbb{R}\\
			\psi(0)=\psi_0,
		\end{array}
	\end{align}
in which $(|\psi|^2\psi)^*$ indicates that the nonlinear interactions are  only taking place on a  quasi-resonant\footnote{The definition of a quasi-resonant set is in  Definition \ref{quasi-res}.} set. Then
\eqref{irrNLS-res} is globally well posed for $\psi_0\in L^2(\mathbb{T}_{\underline{\omega}}^2)$.
\end{theorem}
Here  a generic irrational ordered pair refers to those vectors $\underline{\omega}\in \mathbb{R}^2_+$ whose quotients $\omega^2_1/\omega^2_2$ are not well-approximated by rationals which form a full-measure subset of the real numbers.
\begin{remark}
We note here that in contrast the full NLS initial value problem \eqref{irrNLS} is not known to be well-posed in  $L^2$, not even locally in time. This is  considered a hard open problem. In fact as for now,  local well-posedness for the full NLS equation is available through Strichartz estimates for which $s>0$ is necessary \cite{bourgain1993fourier}.  Below we restate this result  in frequency space in Theorem \ref{lemmaT}, and we  will see in its proof that the reason why one can prove well-posedness for \eqref{irrNLS-res} for $s=0$ is because the irrationality of the torus decouples the components of the frequencies, so that the set up  becomes in a sense  one dimensional, and  in 1D the periodic cubic NLS is globally wellposed in $L^2$ \cite{bourgain1993fourier}.
\end{remark}

The second result of this paper concerns uniform bounds for the norms $H^s, s>1,$  for the solutions to the quasi-resonant problem \eqref{irrNLS-res} that start with a compact support in frequency space. We have the following:

\begin{theorem}\label{mainthm3} Assume that the torus $\mathbb{T}^2_{\underline{\omega}}$ is  a generic irrational ordered pair.
Consider the quasi-resonant NLS  initial value problem \eqref{irrNLS-res}.
 Then if the initial data $\phi_0$ has  compact support in the frequency space,  for any $s>1$, the $H^s$ norm of the associate global solution is uniformly bounded in time.
\end{theorem} 
	
The next step after proving Theorem \ref{mainthm3}	would be to deduce  from it some information about the original NLS \eqref{irrNLS}. Unfortunately at this point we are not able to do so. The main issue is that when one considers large initial data, perturbation type theorems that are also global in time are not in general available.  It may be more suitable to conduct a numerical investigation of this question first, and in fact this will be the content of a future work of the authors.

In addition to the theoretical results above, we also conduct numerical simulations to accompany  the analysis  depicted in this theorem, alongside comparisons with results on a rational torus. The simulations utilize a GPU-accelerated code to integrate the defocusing initial value problem \eqref{irrNLS}  in time on rational or irrational tori. For given initial data size $\|\psi_0\|_s= R$, simulation (or evolution) time $T$ and $\varepsilon$, we numerically find the critical value $M$ such that $\|F^{-1}(\chi_{B_M^c}\hat{\psi}(t))\|_s = \varepsilon $, where $\psi(t)$ is the solution of the full defocusing NLS problem \eqref{irrNLS}, and $\chi_{B_M^c}$ is a smooth cut off function of the complement of the ball $B_M$ centered at the origin and radius $M$. We call this value $M$ the {\it barrier} for the evolution. The value of $M$, as well as the growth of the Sobolev norm and evolution of energy spectrum, are compared between rational and irrational tori. All results consistently show that energy cascades to high frequency is slower on irrational tori than on rational tori, which is further explained in terms of a kinematic analysis on the quasi-resonant sets on the two types of domains.

We are concluding this section by summarizing the structure of the paper. In Section \ref{not} we introduce the notation of  Sobolev spaces for Fourier coefficients. In Section \ref{sec.birkhoff}, we detail the resonances and quasi-resonances of the cubic NLS. Section \ref{sec.normal} quasi-resonant system, including the proof of Theorems   \ref{mainthm2} and \ref{mainthm3}. Finally, Section \ref{numerical} contains a thorough discussion and presentation of the numerical 
simulations.

\section*{Acknowledgement}
A. Hrabski was supported by the NSF grant DGE 1841052,  G. Staffilani was  supported by the NSF grants DMS 1764403, DMS-2052651 and the Simons Foundation, B. Wilson was  supported by NSF  grant  DMS 1856124. G. Staffilani also thanks the organizers of the Abel Symposium that took place in Summer 2023 for the wonderful meeting and the superb hospitality.

\section{Notations}\label{not}

	\subsection{$\ell^2$ Sobolev Spaces}
	
	\begin{definition}
For $x=\{x_n\}_{n \in \mathbb{Z}^2}$ and $s \in [0, \infty)$, define the standard Sobolev norm as
\begin{align*}
\|x\|_s =\|x\|_{h^s}:= \sqrt{ \sum_{n\in \mathbb{Z}^2}|x_n|^2 \langle n \rangle^{2s}   }
\end{align*}
where $\langle n \rangle:= \sqrt{|n|^2+1}$.  Define $h^s(\mathbb{Z}^2)$ as
\begin{align*}
h^s:=h^s(\mathbb{Z}^2):= \left\{ x= \{x_n\}_{n\in \mathbb{Z}^2} \,: \, \|x\|_s<\infty  \right\}.
\end{align*}
\end{definition}
Note that $h^0(\mathbb{Z}^2)=\ell^2(\mathbb{Z}^2)$.

\section{The Quasi-resonant Sets}\label{sec.birkhoff}

 From now on, without loss of generality we assume that  $\underline{\omega}= (1, \omega) \in \mathbb{R}^2_{+}$ is a vector satisfying $(1, \omega^2) \cdot m \neq 0$ for all $ m \in \mathbb{Z}^2$.

 In this section, we start by recalling that our original NLS \eqref{irrNLS} can be viewed as an infinite dimensional Hamiltonian system where the unknowns are the Fourier coefficients  $\hat \psi$ of the original solution $\psi$.  In this frame the associated   infinite dimensional Hamiltonian is
	\begin{align}\label{ham}
		H( \hat{\psi}) &= \tfrac{1}{2} \sum_{k \in \mathbb{Z}^2} \lambda_k |\hat{\psi}_k|^2 + \tfrac{1}{4}\sum_{ k_1+k_2= k_3+k_4} \hat{\psi}_{k_1} \hat{\psi}_{k_2}\bar{\hat{\psi}}_{k_3}\bar{\hat{\psi}}_{k_4}\\\label{HP}
				&=:H_0 + P,
	\end{align}
	where, for $k=(m, \ell) \in \mathbb{Z}^2$,  $\lambda_k:= m^2+\omega^2\ell^2$ are the eigenvalues of the Laplacian.
	In order to conduct our analysis we need to study not only the resonant set, but  also what we define below to be  ``quasi-resonant" sets. These sets will be also analyzed numerically in Section \ref{numerical}. 
	
	\subsection{The Quasi-resonance Condition and the Irrationality of the Torus}

   \begin{definition}\label{quasi-res} For any fixed $\Lambda>0$ and $\tau>1$, we say that any quartet of frequencies, $(\lambda_{k_1}, \lambda_{k_2}, \lambda_{k_3}, \lambda_{k_4})$, or the corresponding quartet of indices, $(k_1, k_2, k_3, k_4)$, are {\bf $(\Lambda, \tau)$-quasi resonant} if they satisfy $k_1+k_2-k_3-k_4=0$ and

	\begin{align}\label{nonres}
	\left|\lambda_{k_1}+\lambda_{k_2}-\lambda_{k_3}-\lambda_{k_4}\right| &\leq \frac{\Lambda}{ \left(|k_{1}|^2+|k_{2}|^2+|k_{3}|^2+|k_{4}|^2\right)^{1+\tau}}
	\end{align}
We  say that a quartet, $(\lambda_{k_1}, \lambda_{k_2}, \lambda_{k_3}, \lambda_{k_4})$, is not quasi-resonant if \eqref{nonres} does not hold.

For completeness we also recall here that if $k_1+k_2+k_3+k_4=0$ and
$\lambda_{k_1}+ \lambda_{k_2}-\lambda_{k_3}- \lambda_{k_4}=0$
then we say that the
quartet, $(\lambda_{k_1}, \lambda_{k_2}, \lambda_{k_3}, \lambda_{k_4})$ is resonant.
\end{definition}

		One can represent the sum, $\lambda_{k_1}+\lambda_{k_2}-\lambda_{k_3}-\lambda_{k_4}$ as $p+\omega^2 q$ where $p$ is the sum of the squares of  the first components of the indices and $q$ represents the sum of the squares of the second components. For arbitrary irrational numbers $\omega^2$, there exists a decreasing function, $g: \mathbb{Z}_+ \rightarrow \mathbb{R}_+$ such that for all $p, q \in \mathbb{Z}_+$
			\begin{align*}
				\left|\frac{p}{q}-\omega^2\right|\geq g(q)
			\end{align*}
			
			If one were to follow the argument  of the last two authors in \cite{staffilani2020stability}, it would become apparent  that the rate of $g$ determines the loss of regularity that one would expect from a  normal form argument  associated with the dispersion relation $\lambda_k$.

		\begin{theorem}\label{gencondition}
		There exists a full measure subset of $\mathbb{R}$, $E,$ such that if $\alpha\in E$, then for any $C>0$, $\tau>0$ there are $N(\alpha,\tau, C)$   finitely many solutions to
			\begin{align}\label{Rothcorineq}
				\left|\frac{p}{q}-\alpha\right| \leq \frac{C}{q^{2+\tau}}
			\end{align}
			for any  $(p, q) \in \mathbb{Z}\times \mathbb{Z}_+$ .
		\end{theorem}
		
		The above theorem \cite{khintchine1924einige} will provide control over the number of $(\Lambda, \tau)$-quasi resonant quartets that are not resonant assuming that $\omega$ belongs to the generic set defined in Theorem \ref{gencondition}. Since we would like to use $\Lambda$ as a parameter, we will need to keep track of the parameter $\Lambda$ on the size of the class  of $(\Lambda, \tau)$-quasi resonant quartets that are not resonant.

	\begin{corollary}\label{quasiroth}
		For any $\Lambda>0$ and $\tau>0$ there exist only finitely many quartets $(k_1, k_2, k_3, k_4)$ satisfying
		\begin{align}\label{freqbound}
			0<\left|\lambda_{k_1}+\lambda_{k_2}-\lambda_{k_3}-\lambda_{k_4}\right| &\leq \frac{\Lambda}{ \left(|k_{1}|^2+|k_{2}|^2+|k_{3}|^2+|k_{4}|^2\right)^{1+\tau}}.
		\end{align}
		Moreover, there exists $C_{\omega}>0$ such that for all $\tau>1$ and $\Lambda>1$, there are no quartets satisfying both
		\begin{align*}
		    \max(|k_1|, |k_2|, |k_3|, |k_4|)>C_{\omega}\Lambda^{\frac{1}{2(\tau-1)}}
		\end{align*}
		and inequality \eqref{freqbound}.
	\end{corollary}
	\begin{proof}
        Let $k_i=(m_i, \ell_i)$ for $i =1, 2, 3, 4$ and note the identity
	    \begin{align*}
	        \lambda_{k_1}+\lambda_{k_2}-\lambda_{k_3}-\lambda_{k_4}&=m_1^2+\omega^2\ell_1^2+m_2^2+\omega^2\ell_2^2-(m_3^2+\omega^2\ell_3^2)-(m_4^2+\omega^2\ell_4^2)\\
	        &=m_1^2+m_2^2-m_3^2-m_4^2+\omega^2(\ell_1^2+\ell_2^2-\ell_3^2-\ell_4^2)
	    \end{align*}
	    Let $p=m_1^2+m_2^2-m_3^2-m_4^2$ and $q=\ell_1^2+\ell_2^2-\ell_3^2-\ell_4^2$, which we can assume positive without loss of generality. Then Theorem \ref{gencondition} implies that for any $\Lambda >0$ and $\tau>0$ there are only finitely many $p \in \mathbb{Z}$ and $q \in \mathbb{Z}_+$ that satisfy
	    \begin{align*}
	        |p-\omega^2 q| \leq \Lambda q^{-(1+\tau)}
	    \end{align*}
	    Since
	    \begin{align*}
	        \ell_1^2+\ell_2^2-\ell_3^2-\ell_4^2 \leq\ell_1^2+\ell_2^2+\ell_3^2+\ell_4^2 \leq |k_{1}|^2+|k_{2}|^2+|k_{3}|^2+|k_{4}|^2
	    \end{align*}
	    we can conclude that there are only finitely many $p \in \mathbb{Z}$ and $q \in \mathbb{Z}_+$ satisfying
	    \begin{align}\label{pqbound}
	        |p-\omega^2 q| \leq \frac{\Lambda}{ \left(|k_{1}|^2+|k_{2}|^2+|k_{3}|^2+|k_{4}|^2\right)^{1+\tau}}
	    \end{align}
	    Now for any fixed $p \in \mathbb{Z}$ and $q \in \mathbb{Z}_+$, there are infinitely many quartets $k_1, k_2, k_3, k_4$ such that $p=m_1^2+m_2^2-m_3^2-m_4^2$ and $q=\ell_1^2+\ell_2^2-\ell_3^2-\ell_4^2$. However, $q \in \mathbb{Z}_+$ implies that $|p-\omega^2 q|>0$, so in order for \eqref{pqbound} to be satisfied it is necessary that
	    \begin{align*}
	        |k_i|\leq \Lambda^{1/(2(1+\tau))} |p-\omega^2 q|^{-1/(2(1+\tau))}
	    \end{align*}
	    for all $i=1, 2, 3, 4$. Therefore, for every fixed $p \in \mathbb{Z}$ and $q \in \mathbb{Z}_+$ there are only finitely many quartets $k_1, k_2, k_3, k_4$ such that \eqref{pqbound} holds. Since there are only finitely many $p$ and $q$ staisfying the Diophantine inequality, \eqref{Rothcorineq}, we have only finitely many quartets satisfying
	    \begin{align*}
	        |p-\omega^2 q| &\leq \frac{\Lambda}{ \left(|k_{1}|^2+|k_{2}|^2+|k_{3}|^2+|k_{4}|^2\right)^{1+\tau}}\\
	        \Leftrightarrow 0<\left|\lambda_{k_1}+\lambda_{k_2}-\lambda_{k_3}-\lambda_{k_4}\right| &\leq \frac{\Lambda}{ \left(|k_{1}|^2+|k_{2}|^2+|k_{3}|^2+|k_{4}|^2\right)^{1+\tau}}
	    \end{align*}
	    The preceding sentence also implies that there exists $M_{\omega, \Lambda, \tau}<\infty$ such that
	    \begin{align*}
	        M_{\omega, \Lambda, \tau}:= \sup\left\{ \max_{i=1, 2, 3, 4} |k_i|~:~ (k_1, k_2, k_3, k_4) \mbox{ satisfies } \eqref{freqbound}\right\}
	    \end{align*}

	 Now assume $\tau>1$ and $\Lambda>1$. From the above argument, there exists a $M_{\omega}:=M_{\omega, 1, 1}$ such that there are no non-trivial solution to the inequality
	 \begin{align*}
		|\lambda_{k_1}+\lambda_{k_2}-\lambda_{k_3}-\lambda_{k_4}| \leq \frac{1}{ \left(|k_{1}|^2+|k_{2}|^2+|k_{3}|^2+|k_{4}|^2\right)^{2}}
	 \end{align*}
	 if $\max_{i=1, 2, 3, 4} |k_i|>M_{\omega, 1, 1}=:M_{\omega}$.
	
	 Furthermore, if $\Lambda\cdot\left(|k_{1}|^2+|k_{2}|^2+|k_{3}|^2+|k_{4}|^2\right)^{-(\tau-1)} \leq 1$ , then
	 \begin{align*}
	 	 \frac{\Lambda}{ \left(|k_{1}|^2+|k_{2}|^2+|k_{3}|^2+|k_{4}|^2\right)^{1+\tau}} \leq \frac{1}{ \left(|k_{1}|^2+|k_{2}|^2+|k_{3}|^2+|k_{4}|^2\right)^{2}}
	 \end{align*}
	which implies that there are no non-trivial solutions to the following inequality
	\begin{align*}
	 	|\lambda_{k_1}+\lambda_{k_2}-\lambda_{k_3}-\lambda_{k_4}|&\leq \frac{\Lambda}{ \left(|k_{1}|^2+|k_{2}|^2+|k_{3}|^2+|k_{4}|^2\right)^{1+\tau}}
	 \end{align*}
	 as long as
	 \begin{align*}
	     \max_{i=1, 2, 3, 4} |k_i|> \max\left( M_{\omega}, \Lambda^{1/2(\tau-1)}  \right).
	 \end{align*}
	 Finally, since $M_{\omega}\Lambda^{1/2(\tau-1)}\geq \max\left( M_{\omega}, \Lambda^{1/2(\tau-1)}  \right)$, it suffices to assume that
	 \begin{align*}
	     \max_{i=1, 2, 3, 4} |k_i|> M_{\omega}\Lambda^{1/2(\tau-1)}.
	 \end{align*}
	 We now set $C_{\omega}:= M_{\omega}$.

	\end{proof}

		The following lemma characterizes the resonant quartets of frequencies. This characterization will be crucial to understanding the dynamics of the normalized system \eqref{reso} below.
	\begin{lemma}\label{axisparallel}
			If
			\begin{align*}
			    \lambda_{k_1}+\lambda_{k_2}-\lambda_{k_3}-\lambda_{k_4}=0
			\end{align*}
			and
			\begin{align*}
			    k_1+k_2-k_3-k_4=0
			\end{align*}
			 then the quartet $(k_1, k_2, k_3, k_4)$ forms an axis parallel rectangle in $\mathbb{Z}^2$ (possibly degenerate).
	\end{lemma}
		
	\begin{proof}
	    First, if  $k_i=(m_i, \ell_i)$ for $i =1, 2, 3, 4$, then
	        \begin{align*}
	            \lambda_{k_1}+\lambda_{k_2}-\lambda_{k_3}-\lambda_{k_4}=0
	        \end{align*}
	   can be written as
	        \begin{align*}
	            &m_1^2+\omega^2\ell_1^2+m_2^2+\omega^2\ell_2^2-(m_3^2+\omega^2\ell_3^2)-(m_4^2+\omega^2\ell_4^2)=0\\
	            &\Leftrightarrow m_1^2+m_2^2-m_3^2-m_4^2+\omega^2(\ell_1^2+\ell_2^2-\ell_3^2-\ell_4^2)=0.
	        \end{align*}
	   Since $\omega^2$ is irrational,
	        \begin{align*}
	           m_1^2+m_2^2-m_3^2-m_4^2&=0  \mbox{ and }\\
	           \ell_1^2+\ell_2^2-\ell_3^2-\ell_4^2&=0.
	        \end{align*}
	   The condition $k_1+k_2-k_3-k_4=0$ implies that $m_1+m_2-m_3-m_4=0$ and $\ell_1+\ell_2-\ell_3-\ell_4=0$. These identities imply that
	        \begin{align*}
	            m_1&=m_3 \mbox{ and } m_2=m_4  \mbox{ or } \\
	            m_1&=m_4 \mbox{ and } m_2= m_3
	        \end{align*}
	        and
	        \begin{align*}
	            \ell_1&=\ell_3 \mbox{ and } \ell_2=\ell_4  \mbox{ or } \\
	            \ell_1&=\ell_4 \mbox{ and } \ell_2= \ell_3
	        \end{align*}
	   In either combination $k_1, k_2, k_3$, and $k_4$ form a (possibly degenerate) axis-parallel rectangle.
	\end{proof}

\section{Dynamics of the Quasi-resonant System} \label{sec.normal}

We now  study the dynamics of the 
system of equations 
	\begin{align}\label{resoNLS}
		\dot{z}_k =i\lambda_k z_k + i\sum_{k=k_1+k_2-k_3}^{**} z_{k_1}z_{k_2}\bar{z}_{k_3} ,
	\end{align}
where the $\sum^{**}$ signifies that the sum is taken over $(\Lambda, \tau)$-quasi-resonant quartets that satisfy
	\begin{align}\label{reso}
	|\lambda_{k_1}+\lambda_{k_2}-\lambda_{k_3}-\lambda_{k_4}| \leq \frac{\Lambda}{ \left(|k_{1}|^2+|k_{2}|^2+|k_{3}|^2+|k_{4}|^2\right)^{1+\tau}}.
	\end{align}
	Briefly focusing on the second  term of the vector field in \eqref{resoNLS}: the conclusion of Corollary \ref{quasiroth} implies that if the maximum of $|k|$, $|k_1|$, $|k_2|$, and  $|k_3|$ is large enough (depending on $\Lambda$ and $\omega$), then the inequality $|\lambda_{k_1}+\lambda_{k_2}-\lambda_{k_3}-\lambda_{k_4}| \leq \Lambda(|k|^2+|k_1|^2+|k_2|^2+|k_3|^2)^{-(1+\tau)}$ does not have any nontrivial solutions. Therefore, we can further decompose the second term in the following way:
	 \begin{align} \label{truncationbound}
	 	&\sum_{k=k_1+k_2-k_3\atop |\lambda_{k_1}+\lambda_{k_2}-\lambda_{k_3}-\lambda_{k}| \leq \Lambda(|k|^2+|k_1|^2+|k_2|^2+|k_3|^2)^{-(1+\tau)}}   z_{k_1}z_{k_2}\bar{z}_{k_3}\\
	 	&= \sum_{k=k_1+k_2-k_3\atop \lambda_{k_1}+\lambda_{k_2}-\lambda_{k_3}-\lambda_{k}=0}   z_{k_1}z_{k_2}\bar{z}_{k_3}
	 	+ \sum_{k=k_1+k_2-k_3 \mbox{ and }|k|^2+|k_1|^2+|k_2|^2+|k_3|^2 <C_{\Lambda, \omega, \tau}\atop 0\neq|\lambda_{k_1}+\lambda_{k_2}-\lambda_{k_3}-\lambda_{k}| \leq \Lambda(|k|^2+|k_1|^2+|k_2|^2+|k_3|^2)^{-(1+\tau)}}   z_{k_1}z_{k_2}\bar{z}_{k_3} \nonumber
	 \end{align}
	 By Corollary \ref{quasiroth}, $C_{\Lambda, \omega, \tau}$ can be set equal to $C_{\omega}\Lambda ^{1/(2(\tau-1))}$.
	
	 \medskip
	
We now restate Theorem \ref{mainthm2} in the frequency space and we prove it.
\begin{theorem}\label{lemmaT}
Let $v_0 = \{v_k\}_{k \in \mathbb{Z}^2} \in h^{s}, s\geq 0$. There exists a unique, global-in-time solution $z \in C^1([0, \infty); h^{s-2}) \cap C([0,\infty);h^{s})$ to the Cauchy problem:
	\begin{equation}\label{star1}
		\left\{ \begin{array}{l}\dot{z}_k =i\lambda_k z_k \pm i\sum_{k=k_1+k_2-k_3}^{**} z_{k_1}z_{k_2}\bar{z}_{k_3} \\
		z(0)=v_0
		  \end{array}  \right.
	\end{equation}
\end{theorem}

\begin{proof}
For the local well-posedness,  we look for the fixed point  in the convex set
	\begin{align*}
		D:= \left\{ z \in C([0, T];h^s) ~:~ z(0)=v_0 \mbox{ and } \sup_{t \in [0, T]}\|z(t)\|_{\ell^2} \leq 10 \|z(0)\|_{\ell^2}\right\}
	\end{align*}
of the map $\mathcal{L}: C([0,T]; h^s) \rightarrow C([0,T]; h^s)$ defined (after a gauge transformation) by

 \begin{align*}\mathcal{L}(z)_k &:= z_k(0) -i\int_0^t d\sigma\,
	 	\sum_{k=k_1+k_2-k_3\atop \lambda_{k_1}+\lambda_{k_2}-\lambda_{k_3}-\lambda_{k}=0}   z_{k_1}z_{k_2}\bar{z}_{k_3}(\sigma)
	 	\\
	 	&+ \sum_{k=k_1+k_2-k_3 \mbox{ and }|k|^2+|k_1|^2+|k_2|^2+|k_3|^2 <C_{\Lambda, \omega,\tau}\atop 0\neq|\lambda_{k_1}+\lambda_{k_2}-\lambda_{k_3}-\lambda_{k}| \leq \Lambda(|k|^2+|k_1|^2+|k_2|^2+|k_3|^2)^{-(1+\tau)}}   e^{i\sigma(\lambda_{k_1}+\lambda_{k_2}-\lambda_{k_3}-\lambda_{k})}z_{k_1}z_{k_2}\bar{z}_{k_3}(\sigma) \nonumber
\end{align*}
We will ignore the estimate of the second sum since by Corollary \ref{quasiroth} only finitely many frequencies are involved in that sum. We want to estimate
    \begin{align*}
        \left\|\sum_{k=k_1+k_2-k_3\atop \lambda_{k_1}+\lambda_{k_2}-\lambda_{k_3}-\lambda_{k}=0}   z_{k_1}z_{k_2}\bar{z}_{k_3}\right\|_{h^s}
   \end{align*}
and for this we use a duality argument. Assume that $\{y_k\}\in \ell^2$, then we estimate
    \begin{align*}
        \sum_k\sum_{k=k_1+k_2-k_3\atop \lambda_{k_1}+\lambda_{k_2}-\lambda_{k_3}-\lambda_{k}=0}   |k|^s z_{k_1}z_{k_2}\bar{z}_{k_3} \bar y_k.
    \end{align*}
 If $k=(m,l)$ then we first approximate  $|k|^s$ with  $[(1+|m|)^s +(1+|\ell|)^s]$ and split the sum
\begin{align*}
		&\sum_{k= (m, \ell) \in \mathbb{Z}^2 \atop}\left(\sum_{k=k_1+k_2-k_3\atop \lambda_{k_1}+\lambda_{k_2}-\lambda_{k_3}-\lambda_{k}=0}  |k|^s z_{k_1}z_{k_2}\bar{z}_{k_3}\bar{y}_k \right)\\
		&\sim \left(\sum_{(m,l)}\sum_{k=k_1+k_2-k_3\atop \lambda_{k_1}+\lambda_{k_2}-\lambda_{k_3}-\lambda_{k}=0}   [(1+|m|)^s +(1+|\ell|)^s]z_{k_1}z_{k_2}\bar{z}_{k_3}\bar{y}_k \right)\\
		&=\left(\sum_{(m,l)}\sum_{k=k_1+k_2-k_3\atop \lambda_{k_1}+\lambda_{k_2}-\lambda_{k_3}-\lambda_{k}=0}   [(1+|m|)^s ]z_{k_1}z_{k_2}\bar{z}_{k_3}\bar{y}_k \right)\end{align*}
	\begin{align*}&\hspace{1cm}+\left(\sum_{(m,l)}\sum_{k=k_1+k_2-k_3\atop \lambda_{k_1}+\lambda_{k_2}-\lambda_{k_3}-\lambda_{k}=0}   [(1+|l|)^s ]z_{k_1}z_{k_2}\bar{z}_{k_3}\bar{y}_k \right)\\	
	\end{align*}
Now we perform a further splitting of the first term thanks to the irrationality of the torus (the second term is handled similarly):
	\begin{align*}
		&\sum_{(m, l) }\sum_{k=k_1+k_2-k_3\atop \lambda_{k_1}+\lambda_{k_2}-\lambda_{k_3}-\lambda_{k}=0}   (1+|m|)^sz_{k_1}z_{k_2}\bar{z}_{k_3}\bar{y}_k \\
		&=\sum_{m}\sum_{m=m_1+m_2-m_3\atop m^2_1+m^2_2-m^2_3-m^2=0}  \sum_{\ell=\ell_1+\ell_2-\ell_3\atop \ell^2_1+\ell^2_2-\ell^2_3-\ell^2=0} (1+|m|)^s z_{(m_1, \ell_1)}z_{(m_2, \ell_2)}\bar{z}_{(m_3, \ell_3)}\bar{y}_{(m, \ell)} \\
		&=\sum_{m}\sum_{m=m_1+m_2-m_3\atop m^2_1+m^2_2-m^2_3-m^2=0}  (1+|m|)^s \sum_{\ell=\ell_1+\ell_2-\ell_3\atop \ell^2_1+\ell^2_2-\ell^2_3-\ell^2=0} z_{(m_1, \ell_1)}z_{(m_2, \ell_2)}\bar{z}_{(m_3, \ell_3)}\bar{y}_{(m, \ell)}
	\end{align*}
	Without loss of generality, fix $(m, m_1, m_2, m_3)$ with $m=m_1$, and $m_2=m_3$. Then
	\begin{align*}
		\sum_{\ell=\ell_1+\ell_2-\ell_3\atop \ell^2_1+\ell^2_2-\ell^2_3-\ell^2=0} z_{(m_1, \ell_1)}z_{(m_2, \ell_2)}\bar{z}_{(m_3, \ell_3)}\bar{y}_{(m, \ell)}
		=\sum_{\ell=\ell_1+\ell_2-\ell_3\atop \ell^2_1+\ell^2_2-\ell^2_3-\ell^2=0} z_{(m, \ell_1)}z_{(m_2, \ell_2)}\bar{z}_{(m_2, \ell_3)}\bar{y}_{(m, \ell)}
	\end{align*}
	and furthermore, there are two possibilities: $\ell=\ell_1$ and $\ell_2=\ell_3$, or $\ell=\ell_2$ and $\ell_1=\ell_3$. In the first case,
	\begin{align*}
		&\sum_{\ell=\ell_1+\ell_2-\ell_3\atop \ell^2_1+\ell^2_2-\ell^2_3-\ell^2=0} z_{(m, \ell_1)}z_{(m_2, \ell_2)}\bar{z}_{(m_2, \ell_3)}\bar{y}_{(m, \ell)}= \sum_{\ell=\ell_1, \ell_2=\ell_3} z_{(m, \ell_1)}z_{(m_2, \ell_2)}\bar{z}_{(m_2, \ell_3)}\bar{y}_{(m, \ell)}\\
		&\leq \left(\sum_{\ell_2} |z_{(m_2, \ell_2)}|^2 \right)\left( \sum_{\ell} |z_{(m, \ell)}|^2 \right)^{1/2}\left( \sum_{\ell} |y_{(m, \ell)}|^2 \right)^{1/2},
	\end{align*}
	
	and to finish we sum in $m_2$ and use Cauchy-Schwarz in $m$.
In the second case,
	\begin{align*}
		&\sum_{\ell=\ell_1+\ell_2-\ell_3\atop \ell^2_1+\ell^2_2-\ell^2_3-\ell^2=0} z_{(m, \ell_1)}z_{(m_2, \ell_2)}\bar{z}_{(m_2, \ell_3)}\bar{y}_{(m, \ell)}= \sum_{\ell=\ell_2, \ell_1=\ell_3} z_{(m, \ell_1)}z_{(m_2, \ell_2)}\bar{z}_{(m_2, \ell_3)}\bar{y}_{(m, \ell)}\\
		&= \left(\sum_{\ell} z_{(m_2, \ell)}\bar{y}_{(m, \ell)} \right)\left( \sum_{\ell_1} z_{(m, \ell_1)}\bar{z}_{(m_2, \ell_1)} \right)\\
		&\lesssim \left(\sum_{\ell} |z_{(m_2, \ell)}|^2\right)^{\frac{1}{2}}
		\left(\sum_{\ell} |y_{(m, \ell)}|^2\right)^{\frac{1}{2}}
		\left(\sum_{\ell_1} |z_{(m, \ell_1)}|^2\right)^{\frac{1}{2}}
		\left(\sum_{\ell_1} |z_{(m_2, \ell_1)}|^2\right)^{\frac{1}{2}}
	\end{align*}
and finally we use Cauchy-Schwarz in $m, m_2$ again. We hence obtain that for all $\{z_k\}\in D$ we have
	\begin{align}\label{sol}&\sup_{[0,T]}\|\mathcal{L}(z)_k\|_{h^s}\leq \|z_k(0)\|_{h^s}+ CT\sup_{[0,T]}\|z_k\|_{h^s}\sup_{[0,T]}\|z_k\|_{\ell^2}^2\\\nonumber&\leq  \|z_k(0)\|_{h^s}+
\sup_{[0,T]}\|z_k\|_{h^s}(10 \|z_k(0)\|_{\ell^2})^2CT.
	\end{align}
Therefore, if $B$ is the ball in $h^s$ centered at the origin with radius $r=2\|z_k(0)\|_{h^s}$ and $T\sim (\|z_k(0)\|_{\ell^2})^{-2}$ small enough, then $\mathcal{L}$ maps the ball centered at origin  and radius $B$ into itself.
Now we want to show that in a smaller interval of time $\mathcal{L}$ is a contraction on the ball $B$. So assume $\{z_k\}\in B$  and $\{w_k\} \in B$ and  take $t \in [-\delta,\delta]$, where $\delta$ will be determined below. Then  one can prove that
    \begin{align*}
        \sup_{[0,T]}\|\mathcal{L}(z)_k-\mathcal{L}(w)_k\|_{h^s}\leq  C\sup_{[0,T]}\|z_k-w_k\|_{h^s}
(\|z_k\|_{h^s} +\|w_k\|_{h^s})^2C\delta
    \end{align*}
and by picking
$\delta \sim r^{-2}$ small enough
one indeed has a contraction in the ball, $B$,  and hence a unique fixed point in the interval $[-\delta,\delta]$.
By repeating the argument in \eqref{sol} we can prove  that on the larger  interval   $[-T,T]$ solutions with initial data $\{z_k(0)\}$ are bounded by $B$. Hence
one can  iterate this argument $T/\delta$ times with the same $B$ to show the existence and uniqueness of solutions in the larger interval $[-T,T]$ since from \eqref{sol} the initial data at any time in this interval is bounded in the norm $h^s$ by the same constant $r$.

In order to iterate to a global in time well-posedness it is enough to prove that the $\ell^2$ norm of solutions to our initial value problem is conserved and hence the same argument as the one above can be repeated in the interval of time $[T,2T]$ and by induction in $[(n-1)T,nT]$ for all $n\in \mathbb{N}$. This would concludes the proof.

In order to show that  the $\ell^2$ norm of solution is conserved we simply show that the square of the $\ell^2$ norm Poisson commutes with the reduced Hamiltonian in the defocusing case, the focusing is the same:
    \begin{align*}
        \frac{d}{dt} \|z\|_{\ell^2}^2
            &=\sum_{k \in \mathbb{Z}^2} \dot{z}_k\bar{z}_k + \sum_{k\in \mathbb{Z}^2} z_k \dot{\bar{z}}_k\\
            &= \sum_k i \lambda_k|z_k|^2+ i \sum_k \left( \sum_{k=k_1+k_2-k_3}^{**} z_{k_1}z_{k_2}\bar{z}_{k_3} \right)\bar{z}_k   \\
             &\hspace{3cm}+ \sum_k -i \lambda_k|z_k|^2- i \sum_k \left( \sum_{k=k_1+k_2-k_3}^{**} \bar{z}_{k_1}\bar{z}_{k_2}z_{k_3} \right)z_k\\
            &=i \sum_k \left( \sum_{k=k_1+k_2-k_3}^{**} z_{k_1}z_{k_2}\bar{z}_{k_3} \right)\bar{z}_k- i \sum_k \left( \sum_{k=k_1+k_2-k_3}^{**} \bar{z}_{k_1}\bar{z}_{k_2}z_{k_3} \right)z_k\\
            &= i \mathscr{L} - i \bar{\mathscr{L}}= i \mathscr{L} - i\mathscr{L}\\
            &=0.
    \end{align*}

\end{proof}

Next we want to show that the dynamics of the quasi-resonant NLS \eqref{resoNLS}, with frequency localized data, takes place  in within a bounded ball. Also in this case we use the defocusing setting, but the focusing is similar.  Before we begin to examine this  dynamics, define the norm on $\mathbb{Z}^2$:
	\begin{align*}
		|k|_{\infty} = \max(|m|, |\ell|) \mbox{ where } k=(m, \ell)
	\end{align*}
	Note that $|k|_{\infty} \sim | k|$.

\begin{definition} \label{cutoff}
	Let $M>0$. Define the cut-off function $V_M:h^s \rightarrow h^s$ by
		\begin{align*}
			[V_M(z)]_k := \left\{ \begin{array}{lc} z_k & |k|_{\infty}>M\\ 0 & |k|_{\infty} \leq M.
			\end{array} \right.
		\end{align*}
\end{definition}

and define an auxiliary cut-off norm:
	\begin{align}\label{N}
		N_M(z)&:=  \sum_{k= (m, \ell) \in \mathbb{Z}^2 \atop |m|>M} [(1+|m|^2)^s +(1+|\ell|^2)^s]|z_k|^2 \\\nonumber
		&\hspace{2cm}+ \sum_{k= (m, \ell) \in \mathbb{Z}^2 \atop |\ell|>M} [(1+|m|^2)^s +(1+|\ell|^2)^s]|z_k|^2
	\end{align}

 We note that since $C_s^{-1} (1+|k|)^{2s} \leq (1+|m|^2)^s +(1+|\ell|^2)^s \leq C_s  (1+|k|)^{2s}$,
 	\begin{align*}
 		C^{-1}_s \|V_Mz\|^2_{h^s} \leq N_M(z) \leq C_s \|V_Mz\|^2_{h^s}
 	\end{align*}

\begin{proposition}\label{nocasc}
	Let  $\Lambda>0$. There exists a $M_0>0$ such that if $M>M_0$ and $y(t)$ is a solution to \eqref{resoNLS}, then
		\begin{align*}
			\frac{d}{dt} N_M(y) =0
		\end{align*}		
\end{proposition}

\begin{proof}

From Corollary \ref{quasiroth}, if we take 
\begin{equation}\label{m0}
M_0\sim C_{\omega}\Lambda^{1/(2(\tau-1))}
\end{equation}
 then for $k=(m, \ell)$, $|m|>M$ or $|\ell|>M$ we have

	\begin{align*}
	 	\dot{y}_k
	 	&= i\lambda_k y_k + i \sum_{k=k_1+k_2-k_3\atop \lambda_{k_1}+\lambda_{k_2}-\lambda_{k_3}-\lambda_{k}=0}   y_{k_1}y_{k_2}\bar{y}_{k_3}\\
	 	&+ i\sum_{k=k_1+k_2-k_3 \mbox{ and }|k|^2+|k_1|^2+|k_2|^2+|k_3|^2 <C_{\omega}^2\Lambda)^{1/(\tau-1)}\atop 0\neq|\lambda_{k_1}+\lambda_{k_2}-\lambda_{k_3}-\lambda_{k}| \leq \Lambda(|k|^2+|k_1|^2+|k_2|^2+|k_3|^2)^{-(1+\tau)}}   y_{k_1}y_{k_2}\bar{y}_{k_3} \\
	 	&=i\lambda_k y_k + i\sum_{k=k_1+k_2-k_3\atop \lambda_{k_1}+\lambda_{k_2}-\lambda_{k_3}-\lambda_{k}=0}   y_{k_1}y_{k_2}\bar{y}_{k_3}
	 \end{align*}

Via direct computation,
	\begin{align*}
		\tfrac{d}{dt} N_M(y(t)) &= \tfrac{d}{dt} \sum_{k= (m, \ell) \in \mathbb{Z}^2 \atop |m|>M} [(1+|m|^2)^s +(1+|\ell|^2)^s]|y_k|^2 \\
		&\hspace{2cm} + \tfrac{d}{dt}\sum_{k= (m, \ell) \in \mathbb{Z}^2 \atop |\ell|>M} [(1+|m|^2)^s +(1+|\ell|^2)^s]|y_k|^2
	\end{align*}

	Due to symmetry, the first term can be represented in the following way
	\begin{align*}
	\tfrac{d}{dt}& \sum_{k= (m, \ell) \in \mathbb{Z}^2 \atop |m|>M} [(1+|m|^2)^s +(1+|\ell|^2)^s]|y_k|^2\\
		&=  \sum_{k= (m, \ell) \in \mathbb{Z}^2 \atop |m|>M} [(1+|m|^2)^s +(1+|\ell|^2)^s](\dot{y}_k\bar{y}_k + y_k \dot{\bar{y}}_k)\\
		&= \sum_{k= (m, \ell) \in \mathbb{Z}^2 \atop |m|>M} \langle|k|\rangle \left(i\lambda_k y_k +i \sum_{k=k_1+k_2-k_3\atop \lambda_{k_1}+\lambda_{k_2}-\lambda_{k_3}-\lambda_{k}=0}   y_{k_1}y_{k_2}\bar{y}_{k_3}\right) \bar{y}_k \\
		&\hspace{2cm}+ \langle|k|\rangle y_k\left( \overline{i\lambda_k y_k+i \sum_{k=k_1+k_2-k_3\atop \lambda_{k_1}+\lambda_{k_2}-\lambda_{k_3}-\lambda_{k}=0}   y_{k_1}y_{k_2}\bar{y}_{k_3}\bar{y}_k} \right)
	\end{align*}
	where $ \langle|k|\rangle:= (1+|m|^2)^s +(1+|\ell|^2)^s$. Simplifying, we obtain
	\begin{align*}
	&\sum_{k= (m, \ell) \in \mathbb{Z}^2 \atop |m|>M} \langle|k|\rangle \left(i\lambda_k y_k +i \sum_{k=k_1+k_2-k_3\atop \lambda_{k_1}+\lambda_{k_2}-\lambda_{k_3}-\lambda_{k}=0}   y_{k_1}y_{k_2}\bar{y}_{k_3}\right) \bar{y}_k\\
	&\hspace{4cm}+ y_k\left( \overline{i\lambda_k y_k+i \sum_{k=k_1+k_2-k_3\atop \lambda_{k_1}+\lambda_{k_2}-\lambda_{k_3}-\lambda_{k}=0}   y_{k_1}y_{k_2}\bar{y}_{k_3}\bar{y}_k} \right) \\
		&=\sum_{k= (m, \ell) \in \mathbb{Z}^2 \atop |m|>M}i\langle|k|\rangle(|\lambda_ky_k|^2-|\lambda_ky_k|^2) \\
		&\hspace{2cm}+\sum_{k= (m, \ell) \in \mathbb{Z}^2 \atop |m|>M}   i \langle|k|\rangle\sum_{k=k_1+k_2-k_3\atop \lambda_{k_1}+\lambda_{k_2}-\lambda_{k_3}-\lambda_{k}=0}   y_{k_1}y_{k_2}\bar{y}_{k_3}\bar{y}_k\\
			&\hspace{4cm} +\sum_{k= (m, \ell) \in \mathbb{Z}^2 \atop |m|>M}   -i \langle|k|\rangle\sum_{k=k_1+k_2-k_3\atop \lambda_{k_1}+\lambda_{k_2}-\lambda_{k_3}-\lambda_{k}=0}   \bar{y}_{k_1}\bar{y}_{k_2}y_{k_3}y_k
	\end{align*}

 The remainder of the computation is the same as that of Theorem \ref{lemmaT}.
\end{proof}

\begin{remark}\label{integrability} As in the proof of Theorem \ref{lemmaT}, the essence of the preceding argument is that the irrationality of the torus leads to a decoupling of the dynamics of the normalized equation into two (or more for higher dimensional tori) disassociated one-dimensional {\bf resonant} cubic Nonlinear Schr\"odinger systems. It is a simple computation to verify that the one-dimensional resonant cubic NLS Hamiltonian Poisson commutes with $h^s$ norms in a similar fashion to what is shown in Theorem \ref{lemmaT} and Proposition \ref{nocasc}. 
\end{remark}

\begin{definition}\label{barr} We call the number $M_0$ introduced in \eqref{m0} the barrier for the quasi-resonant NLS initial value problem \eqref{irrNLS-res}.
\end{definition}

\medskip
We are now ready to prove Theorem \ref{mainthm3}

 \begin{proof} {\bf Proof of Theorem \ref{mainthm3}}. Let $\psi_0$ be such that $ \widehat{\psi_0}$ is supported in a ball of radius $N$. By \eqref{m0}  for the solution $\psi(t)$ we have  $N_{M_0}(\hat \psi(t))=0$ for all $t$ as long as
$M_0= max \{N, C_{\omega}\Lambda^{1/(2(\tau-1))} \}.$ Then we have that
\begin{eqnarray}\label{bound}\|\psi(t)\|_{H^s}^2&\lesssim& M_0^{2s}\|\chi_{M_0}\hat \psi(t)\|_{\l^2}^2+N_{M_0}(\hat \psi(t))\\
\nonumber&\lesssim& M_0^{2s}\|\hat \psi(t)\|_{\l^2}^2\lesssim M_0^{2s}L^2,
\end{eqnarray}
where in the last inequality we used the conservation of the $L^2$-norm, where $\|\psi(0)\|_{L^2}=L^2.$ 

\end{proof}


\section{Numerical Study}\label{numerical}
We now provide the details of our numerical simulation of \eqref{irrNLS} on the rational torus ($\omega^2=1$) and irrational torus ($\omega^2=\sqrt{2}$), as well as our numerical study of  the evolution of the Sobolev norm and energy spectrum on the two tori. We also study numerically  for the full NLS equation \eqref{irrNLS}  the equivalent notion of a barrier according to Definition \ref{barr} above. It is demonstrated through a variety of metrics that the energy cascade process on the irrational torus is consistently slower than on the rational torus.\\

\subsection{Numerical Setup}
We compute the evolution of Fourier coefficients $\hat{\psi}_k$ according to \eqref{irrNLS} via a pseudospectral method on a domain of $512\times512$ modes (with $257\times257$ alias-free modes) \cite{hrabski2020effect}.  The method of integrating factor combined with a 4th order Runge-Kutta scheme is used for the numerical integration, i.e., the linear and nonlinear terms are respectively integrated analytically and explicitly. The frequency definition $\lambda = m^2 + \omega^2\ell^2$ provides explicit dependence on $\omega^2$ for the rational/irrational tori, that appears only at the integrating factor (since the nonlinear term is free of derivatives). We begin with initial conditions
	\begin{align*}
	    \hat{\psi}_{0,k} = \left\{ \begin{array}{lr} C\exp(i\phi_k) & -2 \leq |k|_{\infty} \leq 2, \ k\neq (0,0)\\
	    0 & \text{otherwise},
		\end{array} \right.
	\end{align*}
with $C\in\mathbb{R}$ chosen to yield a prescribed Sobolev norm $R=\|\hat{\psi}_{0}\|_s$ and $\phi_k\in[0,2\pi]$ as uniformly-distributed random phases that are decorrelated with respect to $k$.  In order to make sure that our numerical results are consistent for different realizations of $\phi_k$, we perform 5 simulations starting from different $\phi_k$ realizations and present results considering all realizations.
In the numerical portion of this work, we use Sobolev index $s=2$. The solver is validated by preservation of the Hamiltonian and verifying 4th order convergence with time step $\Delta t$.

We note that an irrational number cannot be exactly represented in our numerical simulation. Instead, our  simulation  of  dynamics  on  an  irrational  torus  is  computed  with $\omega^2$ approximated by a double-precision floating-point number $\omega_*^2$. Nevertheless, a simple analysis shows that the use of $\omega_*^2$ is sufficient to preserve Lemma \ref{axisparallel}, i.e., it does not create spurious resonant quartets other than those in Lemma \ref{axisparallel}. The analysis is shown in Section \ref{App.Finite}.

\subsection{Results}

We begin by generating data for both tori over a range of $R$ for $T=20T_f$, where $T_f=2\pi$ is the fundamental period. For each $R$ and each realization of phase $\phi_k$, $M$, the barrier in this case,  is computed for a range of $\varepsilon$ such that $\|\chi_{B_M^c}\hat{\psi}(T)\|_s = \varepsilon$, where $\psi$ is the solution to the NLS initial value problem \eqref{irrNLS}. The median value of $M$ across the 5 phase realizations is plotted against $R$ for each $\varepsilon$ in Figure \ref{fig:mvr}, with uncertainty bars denoting the maximum and minimum $M$ from the 5 simulations. We find that $M$ consistently takes a smaller value on the irrational torus than on the rational torus, indicating a slower cascade of energy when $\omega^2$ is irrational. We also note that the differences in $M$ between the two tori vary with $R$, taking smaller values near the small and large values of $R$ in our tested range. We hypothesize that for relatively large values of $R$, i.e., the right-most portion of Figure \ref{fig:mvr}, the quasi-resonances on the irrational torus are sufficient to cascade the energy for limited $T$, yielding a small difference in $M$ when compared to the rational torus. For smaller values of $R$, even on the rational torus, the relatively small $T$ is not sufficient for the energy to cascade to significantly larger $M$ than on irrational torus.

\begin{figure}
    \centering
    \includegraphics[scale=1]{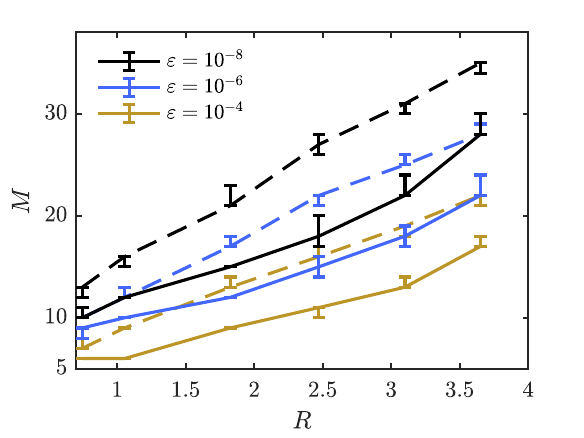}
    \caption{The relationship between $M$ and $R$ for different $\varepsilon$ on log-scale axes for $s=2$ and $T=20T_f$. Data for $\omega^2=\sqrt{2}$ (-----) and $\omega^2=1$ (--~--~--) are included, with lines showing the median value and uncertainty bars showing the maximum and minimum of $M$ among the 5 simulations.}
    \label{fig:mvr}
\end{figure}

We next look at the full energy spectrum to resolve directional differences in evolution on the two tori. The 2D energy spectra of the initial condition and its evolution on both tori are reported in Figure \ref{fig:contour} for $R=1.8263$,  for one realization of $\phi_k$. Aside from an obvious difference in the spread of energy between the tori, we see a higher degree of anisotropy in the spectrum of the irrational torus than that of the rational torus. The irrational torus data develops toward an axes-parallel contour, mainly because of the absence of diamond-type resonant quartets \cite{staffilani2020stability}. We confirm that the rational torus exhibits isotopic evolution for all tested $R$, and that the irrational torus exhibits anisotropic evolution for even the highest tested $R$.

\begin{figure}
    \centering
    \includegraphics[scale=1]{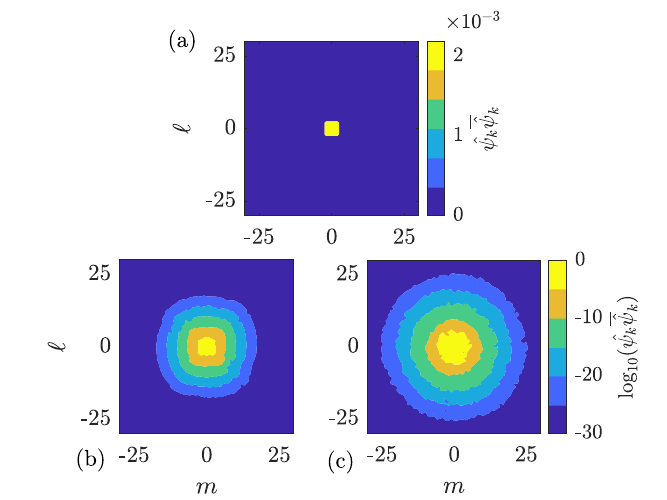}
    \caption{The 2D energy spectra of (a) the initial condition, (b) the irrational torus at $t=20T_f$ and (c) the rational torus at $t=20T_f$. Note that (b) and (c) share the color bar. The zero mode has amplitude $0$, but is colored for simplicity.}
    \label{fig:contour}
\end{figure}

The differences in spectral evolution are also reflected in the growth of the Sobolev norm $\|\hat{\psi}(t)\|_s$, which is plotted over $20T_f$ for all realizations with $R=1.8263$ on both tori in figure \ref{fig:rvt}. The norm on the irrational torus stays close to its initial value $R$, in contrast to the norm on the rational torus, which shows a steady growth. For all tested $R$, $\|\hat{\psi}(t)\|_s$ consistently grows faster on the rational torus than on the irrational torus. Growth of the Sobolev norm for irrational $\omega^2$ is possible in higher regimes of $R$, as depicted in figure \ref{fig:rvt_irr}, which extends to $t=100T_f$. For low $R$, oscillatory behavior of $\|\hat{\psi}(t)\|_s$ is resolved, while for higher $R$, we observe growth of the norm with a rate that increases with $R$. \\

\begin{figure}
    \centering
    \includegraphics[scale=1]{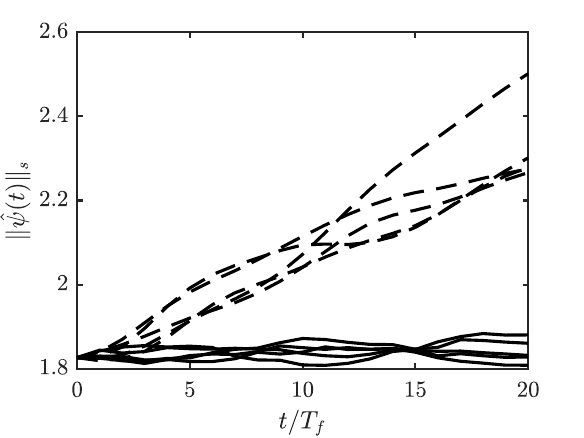}
    \caption{The growth of $\|\hat{\psi}(t)\|_s$ with 5 realizations of $\phi_k$} for $\omega^2=\sqrt{2}$ (-----) and $\omega^2=1$ (--~--~--).
    \label{fig:rvt}
\end{figure}

\begin{figure}
    \centering
    \includegraphics[scale=1]{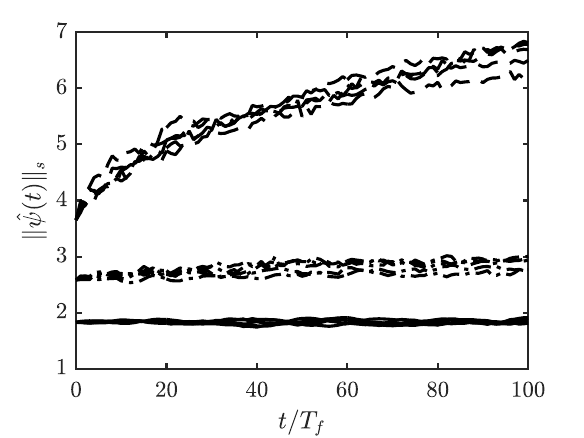}
    \caption{The long-time growth of $\|\hat{\psi}(t)\|_s$ on the irrational torus for $R=3.6526,2.5828,1.8263$ from top to bottom.For each $R$, the growth with $5$ realizations of $\phi_k$ are provided.}
    \label{fig:rvt_irr}
\end{figure}

\subsection{Kinematic Study on $(\Lambda,\tau)$-quasi Resonant Sets}
All results in Section 6.2 indicate that the energy cascade on irrational tori is much less efficient than that on rational tori. This fact can be further understood from a kinematic analysis on the spreading of $(\Lambda,\tau)$-quasi resonant sets on the two tori. By ``kinematic'' we refer to the study which only considers the modes that can be excited by \eqref{nonres} instead of the exact dynamics of \eqref{irrNLS}. Specifically, we choose some initial active modes as the level-1 set, and keep searching for modes that can be excited by $(\Lambda,\tau)$-quasi resonance in the next levels. The detailed numerical algorithm is summarized as follows:
\begin{enumerate}
    \item Fix initial modes in level-1 set $L_1$.
    \item For each 3 modes in $L_1$, find the 4th mode which forms a $(\Lambda,\tau)$-quasi resonant quartet with the 3 modes and is outside $L_1$. Put all such modes in $L_2$.
    \item For each 3 modes in $L_1\cup L_2$, find the 4th mode which forms a $(\Lambda,\tau)$-quasi resonant quartet with the 3 modes and is outside $L_1\cup L_2$. Put all such modes in $L_3$.
    \item Repeat procedure 3 to compute $L_j$ for $j>3$.
\end{enumerate}

We note that this level definition and algorithm differ critically from the one considered in \cite{Colliander-num} which is defined specifically (and only) for resonant sets on the rational torus. Setting $L_1=[-2,2]\times[-2,2]$, Figure \ref{fig:r_ir} shows the results of the quasi-resonant sets for $\tau=0.1$ and three different values of $\Lambda$ on both rational and irrational tori. We compute the sets up to level $N=min(6,\mathcal{N})$, where $\mathcal{N}$ corresponds to the level for which no modes can be excited in the next level. This truncation of $N$ has to be taken because otherwise the levels on the rational torus extends to infinity due to the diamond-type resonance included in the $(\Lambda,\tau)$-quasi resonance.

\begin{figure}
    \centering
    \includegraphics[scale=.7]{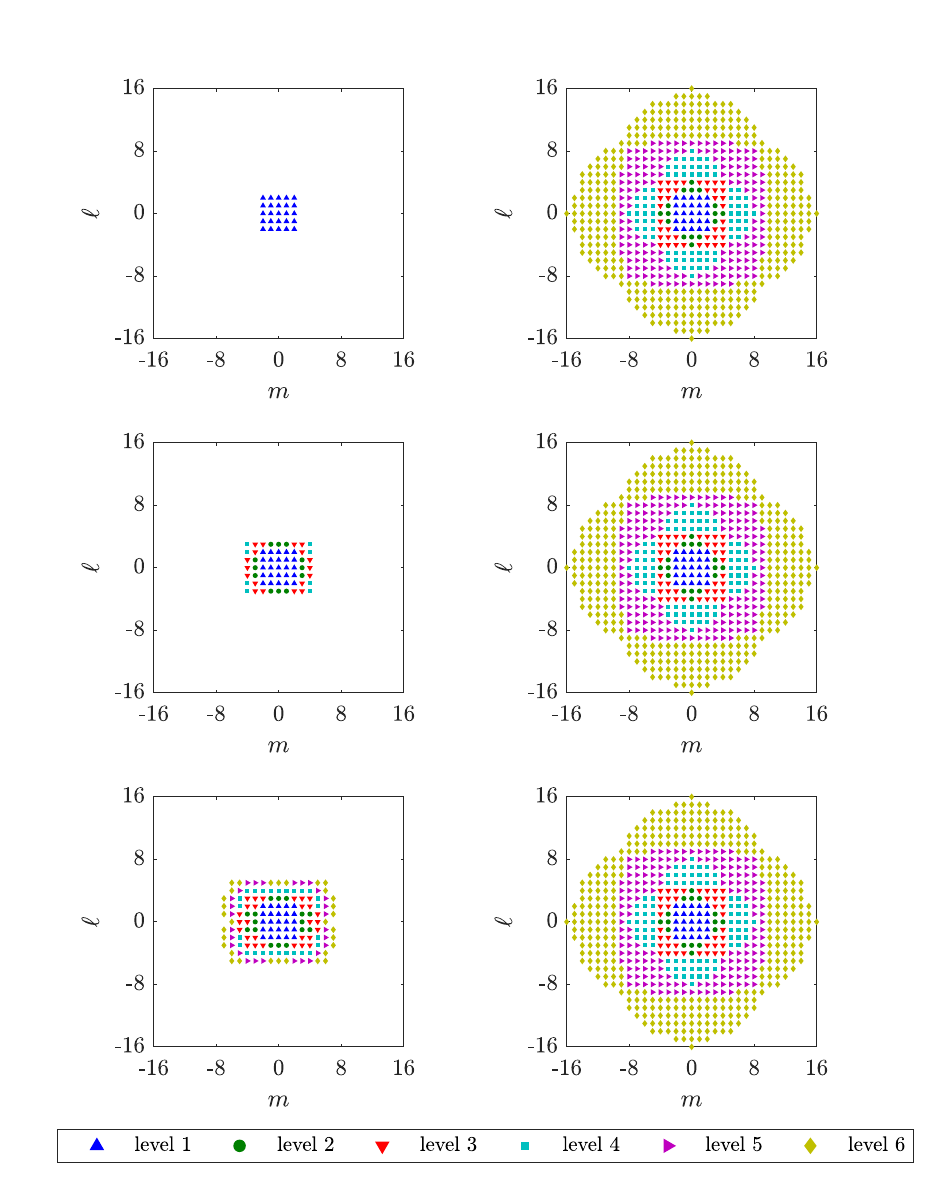}
    \caption{levels of $(\Lambda,\tau)$-quasi resonant sets for $L_1=[-2,2]\times[-2,2]$, $\tau=0.1$, computed up to level $N=min(6,\mathcal{N})$, where $\mathcal{N}$ corresponds to the level for which no modes can be excited in the next level. Left column: $\Lambda=10,20,30$ on irrational torus from top to bottom; Right column: $\Lambda=10,20,30$ on rational torus from top to bottom.}
    \label{fig:r_ir}
\end{figure}

From Figure \ref{fig:r_ir} we see that on the  irrational torus the number of excited modes increase with the value of $\Lambda$. In particular, for $\Lambda=10$, the cascade is disabled without any excited modes; for $\Lambda=20$, the modes extend to level 4 in a rectangular region (the rectangular shape is due to the axis-parallel resonance); for $\Lambda=30$, all six levels are filled with excited modes. In general, the excited modes on the irrational torus is much less in number and more anisotropic compared to those on the rational torus, consistent with our numerical results in \S6.2. Moreover, the $(\Lambda,\tau)$-quasi resonant sets on the rational torus does not depend on $\Lambda$ in the tested range (for larger $\Lambda$ outside the range, it indeed changes the sets). This is because the modal cascade on the rational torus is dominantly controlled by the (diamond-shape) exact resonance included in the $(\Lambda,\tau)$-quasi resonance with any value of $\Lambda$.

\subsection{Discussions on Numerical Study}
We first conclude the numerical study with a summary of results. We study  the barrier of the NLS initial value problem \eqref{irrNLS}  to reveal significant differences in $M$ between tori with $\omega^2=1$ and $\omega^2=\sqrt{2}$, with $M$ taking smaller values on the irrational torus for all tested values of $R$ and $\varepsilon$. We then identify limited spectral growth and anisotropy in the 2D energy spectrum of the irrational torus. These findings are reflected also in the growth of $\|\hat{\psi}(t)\|_s$ on both tori, for which the rational torus growth rate exceeds that of the irrational torus. These findings are consistent with a kinematic analysis on the $(\Lambda,\tau)$-quasi resonant sets on rational and irrational tori. All together, this numerical study demonstrates dramatic differences in the energy cascade capacity of the 2D NLS on rational and irrational tori.

We remark that the results obtained here should place some caveat to the numerics community who study wave turbulence using simulations with periodic boundary conditions. This type of simulations (sometimes referred to as idealized simulations) are widely used in different fields today, such as capillary waves, surface/internal gravity waves and plasma waves. With the rapid growth of computational power, a great effort in the numerics community is to push the resolution of such simulations, utilizing huge computational resources, in order to study the small-scale dynamics. The underlying ``hope'' is that this periodic-domain simulation represents the dynamics on a patch of an infinite domain featuring homogeneous turbulence (see figure \ref{fig:period}). However, due to the results in this work, this ``hope'' may not become true under certain situations. A simple argument can be given as follows: If the periodic-domain simulation provides the same dynamics as the infinite domain, then the results should be consistent for different domain aspect ratios. This is clearly in contradiction with our results in this paper. In fact, how the energy propagates to small scales is critically related to the interaction of dispersion relation and domain aspect ratio, which should be taken into consideration in future applications of idealized simulation for studying wave turbulence.

\begin{figure}[h]
  \centering
  \includegraphics[scale=.3]{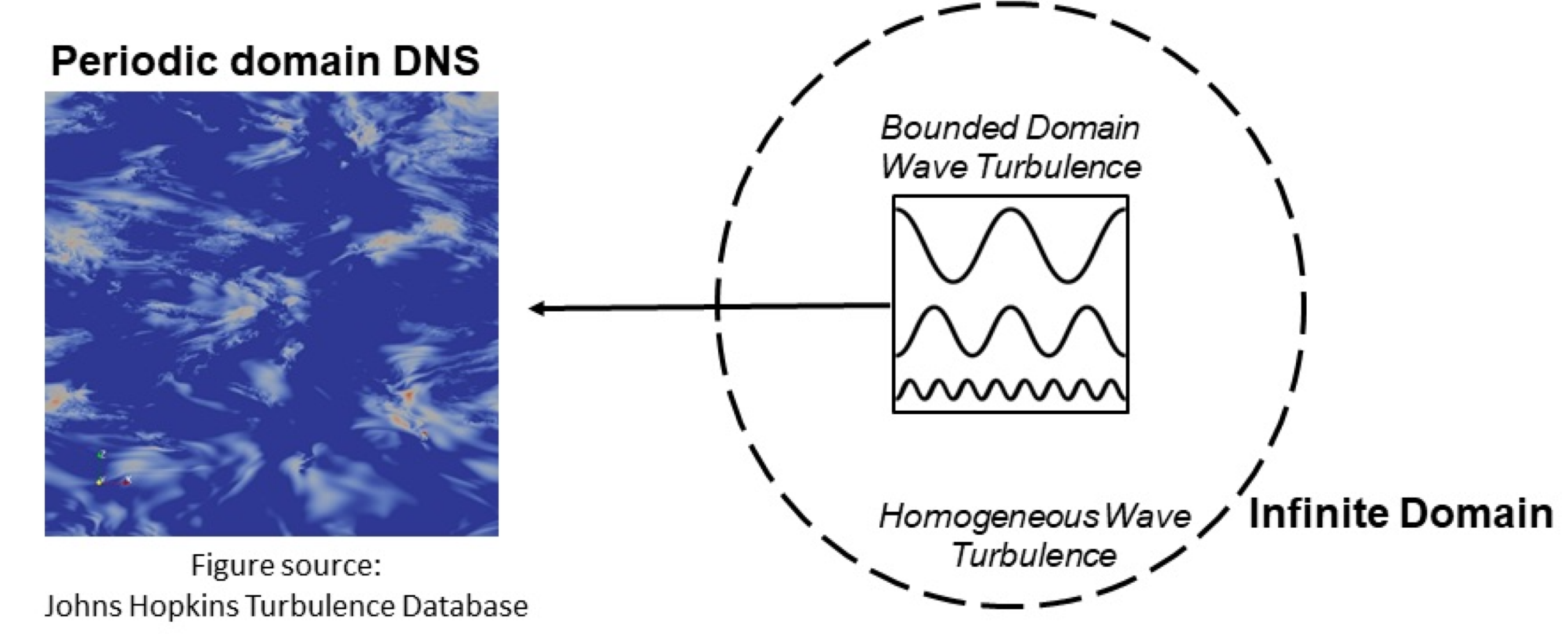}
  \vspace{-6mm}
  \caption{Periodic-domain simulation with the purpose to simulate one path of homogeneous turbulence in an infinite domain.}
  \label{fig:period}
\end{figure}

\pagebreak

\appendix

\section{On the Finite-precision Representation of Irrational $\omega^2$}\label{App.Finite}

We denote a floating-point representation of $\omega^2$ by $\omega_*^2$, with
		\begin{align*}
            \omega_*^2 = \frac{a}{b}
        \end{align*}
        where $a,b\in\mathbb{Z}$ are relatively prime (all floating-point approximations can be written in this way). We consider a quartet $(k_1, k_2, k_3, k_4)$ with
		\begin{align*}
		    \lambda_{k_1}+\lambda_{k_2}-\lambda_{k_3}-\lambda_{k_4}&=0 \mbox{ and }\\
		    k_1+k_2-k_3-k_4&=0.
		\end{align*}
		For $k_i=(m_i, \ell_i)$, recall that the first of these conditions can be written as
        \begin{align*}
            m_1^2+m_2^2-m_3^2-m_4^2+\omega_*^2(\ell_1^2+\ell_2^2-\ell_3^2-\ell_4^2)=0,
        \end{align*}
        while the latter can be expressed as
        \begin{align*}
            m_4&=m_1+m_2-m_3  \mbox{ and }\\
            \ell_4&=\ell_1+\ell_2-\ell_3.
        \end{align*}
        By substituting these conditions on $m_4$ and $\ell_4$ into our first condition built on the dispersion relation, a single quadratic equation can be derived to describe the relationship between the components of our quartet. In its factored form, the equation is
        \begin{align*}
            (m_1-m_3)(m_2-m_3)&=-\omega_*^2(\ell_1-\ell_3)(\ell_2-\ell_3)\\
            \Leftrightarrow \Delta m_{13}\Delta m_{23}&=-\omega_*^2\Delta \ell_{13}\Delta \ell_{23},
        \end{align*}
        where a $\Delta$-notation is adopted for simplicity. As required by Lemma \ref{axisparallel}, equality may only hold for $\Delta m_{13}\Delta m_{23}=-\omega^2\Delta \ell_{13}\Delta \ell_{23}=0$ due to the irrationality of $\omega^2$. We require of our floating-point approximation this same property, namely that
        \begin{align*}
            b\Delta m_{13}\Delta m_{23}=-a\Delta \ell_{13}\Delta \ell_{23}
        \end{align*}
        holds only if the $\text{LHS}=\text{RHS}=0$. We will now identify a necessary condition on any quartet that violates this property. Consider the factorizations of $\Delta m_{13}\Delta m_{23}$ and $\Delta \ell_{13}\Delta \ell_{23}$ and assume
        \begin{align*}
            b\Delta m_{13}\Delta m_{23}=-a\Delta \ell_{13}\Delta \ell_{23} \neq 0.
        \end{align*}
        Because $a$ and $b$ are relatively prime, $\Delta m_{13}\Delta m_{23}$ must have $a$ as a factor and $\Delta \ell_{13}\Delta \ell_{23}$ must have $b$ as a factor. Every quartet that would violate lemma \ref{axisparallel} due to the floating-point approximation of $\omega^2$ has this property. It is now possible to write a condition on $a$ and $b$ such that lemma \ref{axisparallel} is not violated for a given computational domain. Suppose $K=\max(\Delta m_{max},\Delta \ell_{max})$ is the maximum difference in mode number possible on our bounded computational domain. If $K^2<\max(a,b)$ then the only quartets on our domain that satisfy the resonance condition have $\text{LHS}=\text{RHS}=0$, and thus lemma \ref{axisparallel} is preserved.
    
        To obtain $a$ and $b$ in practice, one begins by retrieving the floating-point number $\omega_*^2$. In this case, $\omega_*^2=1.414213562373095$. Then we write $\omega_*^2$ as a fraction and reduce the fraction to relative primes as $a/b$. For example, we can write $\omega_*^2=1414213562373095/1000000000000000$, which reduces to $a=282842712474619$, and $b=200000000000000$. For our simulation of $257\times257$ alias-free modes, $K^2=262144$, so indeed we have $K^2<\max(a,b)$.

        The procedure that determines $a$ and $b$ depends critically on the overall precision of $\omega^2_*$, as the use of single or quad-precision will completely change the final reduced fraction that one obtains due to a change in the number of digits in $\omega^2_*$. Higher precision generally produces a larger $a$ and $b$, and thus would accommodate larger $K$. Conversely, as the size of the computational domain increases, one requires a higher precision approximation of $\omega^2$. In the limit of infinite $K$, no finite-precision (rational) approximation of $\omega^2$ would suffice. We also remark that it is possible to find genuinely rational $\omega^2$ such that the condition $K^2<\max(a,b)$ is met. For our $K^2$, $\omega^2=370723/262144=1.414196014404296875$ is one such example that happens to be close to $\sqrt{2}$, which exhibits irrational-torus behavior in the simulation that we have verified.


\end{document}